\documentclass[reqno]{amsart}
\usepackage{hyperref}
\usepackage{amssymb}

\allowdisplaybreaks

\theoremstyle{plain}
\newtheorem{theorem}{Theorem}[section]

\newtheorem{lemma}[theorem]{Lemma}
\newtheorem{corollary}[theorem]{Corollary}

\theoremstyle{definition}
\newtheorem{definition}[theorem]{Definition}

\theoremstyle{remark}
\newtheorem{remark}[theorem]{Remark}
\numberwithin{equation}{section}

\begin{document}

\title[Parabolic equation with nonlinear nonlocal boundary condition]
{Initial boundary value problem for a semilinear parabolic
equation with absorption and nonlinear nonlocal boundary condition}

\author[A. Gladkov]{Alexander Gladkov}
\address{Alexander Gladkov \\ Department of Mechanics and Mathematics
\\ Belarusian State University \\ Nezavisimosti Avenue 4, 220030
Minsk, Belarus} \email{gladkoval@mail.ru}


\begin{abstract}
In this paper we consider an initial boundary value problem for a
semilinear parabolic equation with absorption and  nonlinear
nonlocal Neumann boundary condition. We prove comparison
principle, the existence theorem of a local solution and  study
the problem of uniqueness and nonuniqueness.
\end{abstract}

\maketitle

\section{Introduction}

In this paper we consider the initial boundary value problem for
the following semilinear parabolic equation
\begin{equation}\label{v:u}
    u_t=\Delta u-c(x,t)u^p,\;x\in\Omega,\;t>0,
\end{equation}
with nonlinear nonlocal boundary condition
\begin{equation}\label{v:g}
\frac{\partial u(x,t)}{\partial\nu}=
\int_{\Omega}{k(x,y,t)u^l(y,t)}\,dy,\;x\in\partial\Omega, \; t
\geq 0,
\end{equation}
and initial datum
\begin{equation}\label{v:n}
    u(x,0)=u_{0}(x),\; x\in\Omega,
\end{equation}
where $p>0,\,l>0$, $\Omega$ is a bounded domain in $\mathbb{R}^n$
for $n\geq1$ with smooth boundary $\partial\Omega$, $\nu$ is unit
outward normal on $\partial\Omega.$

Throughout this paper we suppose that the functions
$c(x,t),\;k(x,y,t)$ and $u_0(x)$ satisfy the following conditions:
\begin{equation*}
c(x,t)\in
C^{\alpha}_{loc}(\overline{\Omega}\times[0,+\infty)),\;0<\alpha<1,\;c(x,t)\geq0;
\end{equation*}
\begin{equation*}
k(x, y, t)\in
C(\partial\Omega\times\overline{\Omega}\times[0,+\infty)),\;k(x,y,t)\geq0;
\end{equation*}
\begin{equation*}
u_0(x)\in C^1(\overline{\Omega}),\;u_0(x)\geq0\textrm{ in
}\Omega,\;\frac{\partial u_0(x)}{\partial\nu}=\int_{\Omega}{k(x,
y,0)u_0^l(y)}\,dy\textrm{ on }\partial\Omega.
\end{equation*}
A lot of articles have been devoted to the investigation of
initial boundary value problems for parabolic equations and
systems with nonlinear nonlocal Dirichlet boundary condition (see,
for example,~\cite{Deng, Fang, Gao, Gladkov_Guedda,
Gladkov_Guedda2, Gladkov_Kim, Gladkov_Kim1, Gladkov_Nikitin, Liu1,
Liu2, Zhong, Zhou} and the references therein). In particular, the
initial boundary value problem for equation~(\ref{v:u}) with
nonlocal boundary condition
\begin{equation*}
    u(x,t)=\int_{\Omega}k(x,y,t)u^l(y,t)\,dy,\;x\in\partial\Omega,\;t>0,
\end{equation*}
was considered for $c(x,t) \leq 0$ and $c(x,t) \geq 0$
in~\cite{Gladkov_Guedda, Gladkov_Guedda2} and~\cite{Gladkov_Kim,
Gladkov_Kim1} respectively. The problem~(\ref{v:u})--(\ref{v:n})
with $c(x,t) \leq 0$ were investigated in~\cite{Gladkov_Kavitova,
Gladkov_Kavitova2}.

We note that for $p<1$ and $l<1$ the nonlinearities in
equation~(\ref{v:u}) and boundary condition~(\ref{v:g}) are
non-Lipschitzian. The problem of uniqueness and nonuniqueness for
different parabolic nonlinear equations with non-Lipschitzian data
in bounded domain has been addressed by several authors (see, for
example, \cite{B, Cortazar, CER, Escobedo_Herrero, FW,
Gladkov_Guedda2, Gladkov_Kim1, K} and the references therein).

The aim of this paper is to study problem~(\ref{v:u})--(\ref{v:n})
for any $p>0$ and $l>0.$ We prove existence of a local solution
and establish some uniqueness and nonuniqueness results.

This paper is organized as follows. In the next section we prove
the existence of a local solution. Comparison principle and  the
problem of uniqueness and nonuniqueness
for~(\ref{v:u})--(\ref{v:n}) are investigated in Section~3.

\section{ Local existence}

In this section a local existence theorem
for~(\ref{v:u})--(\ref{v:n}) will be proved. We begin with
definitions of a supersolution, a subsolution and a maximal
solution of~(\ref{v:u})--(\ref{v:n}). Let
$Q_T=\Omega\times(0,T),\;S_T=\partial\Omega\times(0,T)$,
$\Gamma_T=S_T\cup\overline\Omega\times\{0\}$, $T>0$.
\begin{definition}\label{v:sup}
We say that a nonnegative function $u(x,t)\in C^{2,1}(Q_T)\cap
C^{1,0}(Q_T\cup\Gamma_T)$ is a supersolution
of~(\ref{v:u})--(\ref{v:n}) in $Q_{T}$ if
        \begin{equation}\label{v:sup^u}
u_{t}\geq\Delta u-c(x, t)u^{p},\;(x,t)\in Q_T,
        \end{equation}
        \begin{equation}\label{v:sup^g}
\frac{\partial u(x,t)}{\partial\nu}\geq\int_{\Omega}{k(x, y,
t)u^l(y, t) }\,dy, \; x \in \partial \Omega,\; 0 \leq t < T,
        \end{equation}
        \begin{equation}\label{v:sup^n}
            u(x,0)\geq u_{0}(x),\; x\in\Omega,
        \end{equation}
and $u(x,t)\in C^{2,1}(Q_T)\cap C^{1,0}(Q_T\cup\Gamma_T)$ is a
subsolution of~(\ref{v:u})--(\ref{v:n}) in $Q_{T}$ if $u\geq0$ and
it satisfies~(\ref{v:sup^u})--(\ref{v:sup^n}) in the reverse
order. We say that $u(x,t)$ is a solution of
problem~(\ref{v:u})--(\ref{v:n}) in $Q_T$ if $u(x,t)$ is both a
subsolution and a supersolution of~(\ref{v:u})--(\ref{v:n}) in
$Q_{T}$.
\end{definition}
\begin{definition}\label{v:max1}
We say that a solution $u(x,t)$ of~(\ref{v:u})--(\ref{v:n}) in
$Q_{T}$ is a maximal solution if for any other solution $v(x,t)$
of~(\ref{v:u})--(\ref{v:n}) in $Q_{T}$ the inequality $v(x,t)\leq
u(x,t)$ is satisfied for $(x,t)\in Q_T\cup\Gamma_T$.
\end{definition}
\begin{definition}\label{Def2}
We say that $u$ is a strict supersolution of problem
(\ref{v:u})--(\ref{v:n}) in $Q_T$ if it is a supersolution in
$Q_T$ and the inequality in (\ref{v:sup^g}) is strict. Analogously
we say that $u$ is a strict subsolution of problem
(\ref{v:u})--(\ref{v:n}) in $Q_T$ if it is a subsolution in $Q_T$
and the inequality in (\ref{v:sup^g}) is reversed and strict.
\end{definition}

Let $\{\varepsilon_m\}$ be decreasing to $0$ sequence such that
$0<\varepsilon_m<1.$ For $\varepsilon=\varepsilon_m$ let
$u_{0\varepsilon}(x)$ be the functions with the following
properties: $u_{0\varepsilon}(x) \in C^1(\overline\Omega),\,$
$u_{0\varepsilon}(x) \ge \varepsilon,\,$ $u_{0\varepsilon_i}(x)
\ge u_{0\varepsilon_j}(x)$ for $\varepsilon_i>\varepsilon_j, \,$
$u_{0\varepsilon}(x) \to u_{0}(x)$ as $\varepsilon \to 0$ and
$$
\frac{\partial u_{0\varepsilon} (x)}{\partial\nu} = \int_{\Omega}
k(x,y,0) u^l_{0\varepsilon}(y) \, dy,
$$
for $x \in \partial \Omega.$ Since the nonlinearities in
(\ref{v:u}), (\ref{v:g}) , the Lipschitz condition can be not
satisfied, and thus we need to consider the following auxiliary
problem:
\begin{eqnarray} \label{E:2.1}
\left\{ \begin{array}{ll} u_{t}= \Delta u - c(x,t) u^p + c(x,t)
\varepsilon^p \,\,\,&\textrm{for} \,\,\, x \in \Omega,
\,\,\,\,\, t > 0, \\
\frac{\partial u(x,t)}{\partial\nu} = \int_{\Omega} k(x,y,t) u^l
(y,t) dy \,\,\,
& \textrm{for} \,\,\, x \in \partial \Omega, \,\, t \geq 0,  \\
u(x,0)= u_{0\varepsilon}(x)  \,\,\,& \textrm{for} \,\,\, x \in
\Omega,
\end{array} \right.
\end{eqnarray}
where $\varepsilon=\varepsilon_m.$  The notion of a solution
$u_\varepsilon$ for problem (\ref{E:2.1}) can be defined in a
similar way as in the Definition~\ref{v:sup}.
\begin{theorem}\label{Th1} Problem (\ref{E:2.1})
has a unique solution in $Q_T$ for small values of $T.$
\end{theorem}
\begin{proof}
We start the proof with the construction of a supersolution of
(\ref{E:2.1}). Let $\sup_{\Omega} u_{0\varepsilon}(x) \leq M.$
Denote $ K=\sup_{\partial \Omega \times Q_1}k(x,y,t)$ and
introduce an auxiliary function $\psi (x)$ with the following
properties:
 $$
 \psi (x)\in C^2(\overline{\Omega}),\,
 \inf_{\Omega}\psi (x) \ge 1, \,
\inf_{\partial\Omega}\frac{\partial\psi (x)}{\partial \nu}  \ge
KM^{l-1}\max\{1,\exp(l-1)\} \int_{\Omega}\psi^l(y)dy.
$$
Set $\alpha = \sup_{\Omega}\Delta \psi (x).$ Then it is not
difficult to check that
$$
w(x,t)=M\exp (\alpha t)\psi (x)
$$
is a supersolution of (\ref{E:2.1}) in $Q_T$ if $T  \leq \min \{
1/\alpha,1 \}.$

To prove the existence of a solution for (\ref{E:2.1}) we
introduce the set
$$
B = \{ h(x,t) \in C(\overline{Q_T}): \varepsilon \leq h(x,t) \leq
w(x,t), \, h(x,0)= u_{0\varepsilon}(x)\}
$$
and consider the following problem
\begin{eqnarray} \left\{ \begin{array}{ll}
u_{t}= \Delta u - c(x,t) u^p + c(x,t) \varepsilon^p
\,\,\,&\textrm{for} \,\,\, x \in \Omega,
\,\,\,\,\, 0<t<T,  \\
\frac{\partial u(x,t)}{\partial\nu}  = \int_{\Omega} k(x,y,t) v^l
(y,t) dy  \,\,\,
& \textrm{for} \,\,\, x \in \partial \Omega, \,\, 0 \leq t<T, \label{E:2.2} \\
u (x,0)= u_{0\varepsilon}(x) \,\,\,& \textrm{for} \,\,\, x \in
\Omega,
\end{array} \right.
\end{eqnarray}
where $ v \in B.$ It is obvious, $B$ is a nonempty convex subset
of $C(\overline{Q_T}).$ By classical theory \cite{LSU} problem
(\ref{E:2.2}) has a solution $u \in C^{2,1}(Q_T) \cap
C^{1,0}(\overline{Q}_T ).$ Let us call $Av=u.$ In order to show
that $A$ has a fixed point in $B$ we verify that $A$ is a
continuous mapping from $B$ into itself such that $AB$ is
relatively compact. Thanks to a comparison principle for
(\ref{E:2.2}) we have that $A$ maps $B$ into itself.

Let $G(x,y;t)$ denote the Green's function for a heat equation
given by
$$
u_t - \Delta u =0 \,\,\, \textrm{for}\,\,\, x \in \Omega, \, t>0
$$
with homogeneous Neumann boundary condition. Then $u(x,t)$ is a
solution of (\ref{E:2.2}) in $Q_T$ if and only if
\begin{eqnarray}\label{E:2.5}
\hspace{-0.7cm} u(x,t) &=& \int_\Omega G(x,y;t)
u_{0\varepsilon}(y) \, dy + \int_0^t\int_{\Omega} G(x,y;t-\tau)
c(y,\tau) (\varepsilon^p - u^p(y,\tau) )\, dy d\tau
\nonumber\\
&+&\int_0^t\int_{\partial\Omega}  G (x,\xi;t-\tau) \int_{\Omega}
k(\xi,y,\tau) v^l (y,\tau)\, dy  \, dS_{\xi} d\tau
\end{eqnarray}
for $(x,t) \in Q_T.$

We claim that $A$ is continuous. In fact let $v_k$  be a sequence
in $B$ converging to $v \in B$ in $C(\overline{Q_T}).$ Denote
$u_k=Av_k.$ Then by (\ref{E:2.5}) we see that
\begin{eqnarray*}
|u-u_k|  &\leq& \theta \sup_{Q_T} |u-u_k| \int_0^t\int_{\Omega}
G(x,y;t-\tau)
c(y,\tau)\, dy d\tau \\
&+& \sup_{Q_T} |v^l-v^l_k| \int_0^t\int_{\partial\Omega}
 G (x,\xi;t-\tau) \int_{\Omega}
k(\xi,y,\tau) \, dy dS_{\xi} d\tau ,
\end{eqnarray*}
where $\theta = p \max \{\varepsilon^{p-1},\sup_{Q_T}w^{p-1} \}.$
We note that (see \cite{Hu_Yin1})
$$
\theta \sup_{Q_T}\int_0^t\int_{\Omega} G(x,y;t-\tau) c(y,\tau)\,
dy d\tau < 1
$$
for small values of $T.$ Now we can conclude that $u_k$ converges
to $u_{\varepsilon}$ in $C(\overline{Q_T}).$

The equicontinuity of $AB$  follows from (\ref{E:2.5}) and the
properties of the Green's function (see, for example,
\cite{PaoBook}). The Ascoli-Arzel\'a theorem guarantees the
relative compactness of $AB.$ Thus we are able to apply the
Schauder-Tychonoff fixed point theorem and conclude that $A$ has a
fixed point in $B$ if $T$ is small. Now if $u_{\varepsilon}$ is a
fixed point of $A,\,$ $u_{\varepsilon} \in C^{2,1}(Q_T) \cap
C^{1,0}(\overline{Q}_T )$ and it is a solution of (\ref{E:2.1}) in
$Q_T.$ Uniqueness of the solution follows from a comparison
principle for (\ref{E:2.1}) which can be proved in a similar way
as in the next section.
\end{proof}

Now, let $\varepsilon_2 > \varepsilon_1.$  Then it is easy to see
that $u_{\varepsilon_2}(x,t)$ is a supersolution of problem
(\ref{E:2.1}) with $\varepsilon = \varepsilon_1.$ Applying to
problem (\ref{E:2.1}) a comparison principle we have $
u_{\varepsilon_1}(x,t) \leq u_{\varepsilon_2}(x,t).$ Using the
last inequality and the continuation principle of solutions we
deduce that the existence time of $u_\varepsilon$ does not
decrease as $\varepsilon \searrow 0.$ Taking $\varepsilon \to 0$,
we get $u_m(x,t)=\lim_{\varepsilon \to 0} u_\varepsilon(x,t)\geq
0$ and $u_m(x,t)$ exists in $Q_T$ for some $T>0.$ By dominated
convergence theorem, $u_m(x,t)$ satisfies the following equation
\begin{eqnarray*}
\hspace{-0.7cm}  u_m(x,t) &=& \int_\Omega G(x,y;t) u_{0}(y) \, dy
- \int_0^t\int_{\Omega} G(x,y;t-\tau) c(y,\tau) u_m^p(y,\tau) )\,
dy d\tau\\
&+&\int_0^t\int_{\partial\Omega} G (x,\xi;t-\tau) \int_{\Omega}
k(\xi,y,\tau) u_m^l (y,\tau)\, dy d\xi d\tau
\end{eqnarray*}
for $(x,t) \in Q_T.$ Further, the interior regularity of
$u_m(x,t)$ follows from the continuity of $u_m(x,t)$ in $Q_T$ and
the properties of the Green's function. Obviously, $u_m(x,t)$
satisfies (\ref{v:u})--(\ref{v:n}).
\begin{theorem}\label{Th2}
Problem (\ref{v:u})--(\ref{v:n}) has a solution in $Q_T$ for small
values of  $\,T.$
\end{theorem}

\section{Uniqueness and nonuniqueness}\label{cp}
We start this section with a comparison principle for problem
(\ref{v:u})--(\ref{v:n}) which is used below.

\begin{theorem}\label{Th3} Let $\overline{u}$ and $\underline{u}$ be a
 supersolution and a  subsolution of problem
(\ref{v:u})--(\ref{v:n}) in $Q_T,$ respectively. Suppose that
$\underline{u}(x,t)> 0$ or $\overline{u}(x,t) > 0$ in ${Q}_T\cup
\Gamma_T$ if $l < 1$. Then $ \overline{u}(x,t) \geq
\underline{u}(x,t) $ in ${Q}_T\cup \Gamma_T.$
\end{theorem}
\begin{proof}
Suppose that $l \geq 1.$ Let $T_0 \in (0,T)$ and
$u_{0\varepsilon}(x)$ have the same properties as in previous
section but $u_{0\varepsilon}(x) \to \underline{u} (x,0)$ as
$\varepsilon \to 0.$ We can construct a solution $u_m(x,t)$ of
(\ref{v:u})--(\ref{v:n}) with $u_0(x)=\underline{u} (x,0)$ in the
following way $u_m(x,t)=\lim_{\varepsilon \to 0}
u_\varepsilon(x,t)$ where $u_\varepsilon(x,t)$ is a solution of
(\ref{E:2.1}). To establish theorem we shall show that
\begin{equation}\label{E:1.1}
\underline{u}(x,t) \leq u_m(x,t) \leq \overline{u}(x,t) \,\,\,
 \textrm{in} \,\,\, \overline{Q}_{T_0}  \,\,\,
 \textrm{for any} \,\,\, T_0.
\end{equation}
We prove the second inequality in (\ref{E:1.1}) only since the
proof of the first one is similar. Let $\varphi (x,t) \in
C^{2,1}(\overline{Q}_{T_0})$ be a nonnegative function such that
$$
\frac{\partial \varphi (x,t)}{\partial \nu} = 0, \; (x,t) \in
S_{T_0}.
$$
If we multiply the first equation in (\ref{E:2.1}) by $\varphi
(x,t)$ and then integrate over $Q_{t}$ for $ 0 < t < T_0,$ we get
\begin{eqnarray}\label{E:1.3}
\hspace{-0.7cm}\int_\Omega u_\varepsilon(x,t) \varphi(x,t)\, dx
&\leq& \int_\Omega u_\varepsilon(x,0)\varphi (x,0)\, dx +
\varepsilon^p \int_0^t\int_{\Omega} c(x,\tau) \varphi \, dx d\tau
\nonumber\\
&+& \int_0^t\int_{\Omega} (u_\varepsilon \varphi_{\tau} +
u_\varepsilon \Delta \varphi - c(x,\tau) u_\varepsilon^p \varphi)
\,
dx d\tau \nonumber \\
&+&\int_0^t\int_{\partial\Omega} \varphi (x,\tau) \int_\Omega
k(x,y,\tau) u_\varepsilon^l(y,\tau)\, dy  \, dS_x d\tau,
\end{eqnarray}
On the other hand, $\overline{u}$ satisfies (\ref{E:1.3}) with
reversed inequality and with $\varepsilon =0.$ Set
$w(x,t)=u_\varepsilon(x,t) - \overline{u}(x,t).$ Then $w(x,t)$
satisfies
\begin{eqnarray}\label{E:1.5}
\hspace{-0.9cm}\int_\Omega w(x,t)\varphi (x,t)\, dx &\leq&
\int_\Omega w(x,0)\varphi (x,0)\, dx + \varepsilon^p
\int_0^t\int_{\Omega} c(x,\tau) \varphi \, dx d\tau\nonumber\\
&+& \int_0^t\int_{\Omega} (\varphi_{\tau} + \Delta \varphi -
 c(x,\tau) p\theta_1^{p-1} \varphi)) w \, dx d\tau \nonumber \\
&+& \int_0^t\int_{\partial\Omega} \varphi (x,\tau) \int_\Omega
k(x,y,\tau) l \theta_2^{l-1}w(y,\tau)\, dy dS_x d\tau, \,
\end{eqnarray}
where  $\theta_1$ and $\theta_2$ are some continuous functions
between $u_\varepsilon$ and $\overline{u}.$ Note here that by
hypotheses for $c(x,t)$, $k(x,y,t)$, $u_\varepsilon(x,t)$ and
$\overline{u}(x,t)$, we have
\begin{eqnarray}\label{E:1.6}
&&0 \leq c(x,t) \leq M, \,\,\,  0 \leq \overline{u}(x,t) \leq M
\,\,\, \varepsilon \leq u_\varepsilon(x,t) \leq M \,\,\,
\textrm{in}\,\,  \overline{Q}_{T_0} \nonumber \\
&&\textrm{and}\,\,\,0 \leq k(x,y,t) \leq M \,\,\, \textrm{in}\,\,
\partial \Omega \times \overline{Q}_{T_0},
\end{eqnarray}
where $M$ is some positive constant. Then, it is easy to see from
(\ref{E:1.6}) that $\theta_1^{p-1}$ and $\theta_2^{l-1}$ are
positive and bounded functions in $\overline{Q}_{T_0}$ and
moreover, $\theta_2^{l-1} \leq M^{l-1}.$ Define a sequence $\{a_n
\}$ in the following way: $\, a_n \in C^\infty
(\overline{Q}_{T_0}),\,$ $ a_n \geq 0 \,$ and $\, a_n \to c(x,t)p
\theta_1^{p-1}\,$ as $\, n \to \infty \,$ in $L^1({Q}_{T_0}).$
Now, consider a backward problem given by
\begin{eqnarray}\label{E:1.7}
\left\{ \begin{array}{ll} \varphi_{\tau} + \Delta \varphi - a_n
\varphi = 0 \,\,\, & \textrm{for} \,\,\,x \in \Omega, \,\,\,
0<\tau<t, \\
\frac{\partial \varphi (x,\tau)}{\partial \nu} = 0 \,\,\,&
\textrm{for} \,\,\, x \in
\partial\Omega, \,\,\, 0 \leq \tau < t, \\
\varphi(x,t)= \psi (x) \,\,\,& \textrm{for} \,\,\,x \in \Omega,
\end{array} \right.
\end{eqnarray}
where $\psi (x)\in C_0^\infty (\Omega)$ and $ 0 \leq \psi (x) \leq
1.$ Denote a solution of (\ref{E:1.7}) as $\varphi_n (x,\tau).$
Then by the standard theory for linear parabolic equations (see
\cite{LSU}, for example), we find that  $\varphi_n \in
C^{2,1}(\overline{Q}_{t}),$  $0 \leq \varphi_n (x,\tau) \leq 1$ in
$\overline{Q}_{t}.$  Putting $\varphi = \varphi_n$ in
(\ref{E:1.5}) and passing then to the limit as $n \to \infty$ we
infer
\begin{equation}\label{E:1.8}
\int_\Omega w(x,t)\psi (x)\, dx \leq \int_\Omega w(x,0)_+ \, dx +
\varepsilon^p M  T_0 |\Omega| +  l M^l |\partial\Omega| \int_0^t
\int_\Omega w(y,\tau)_+ \, dy d\tau,
\end{equation}
where $w_+ = \max \{w,0 \},$ $|\partial\Omega|$ and $|\Omega|$ are
the Lebesgue measures of $\partial\Omega$ in $\mathbb R^{n-1}$ and
$\Omega$ in $\mathbb R^n,$ respectively. Since (\ref{E:1.8}) holds
for every $\psi (x),$ we can choose a sequence $\{ \psi_n \}$
converging on $\Omega$ to $\psi (x) =1$ if $w(x,t) > 0$ and $\psi
(x) = 0$ otherwise. Hence, from (\ref{E:1.8}) we get
\begin{equation*}
\int_\Omega w(x,t)_+ \, dx \leq \int_\Omega w(x,0)_+ \, dx +
\varepsilon^p M  T_0 |\Omega| + l M^l |\partial\Omega| \int_0^t
\int_\Omega w(y,\tau)_+ \, dy d\tau.
\end{equation*}
Applying now Gronwall's inequality and passing to the limit
$\varepsilon \to 0,$ the conclusion of this theorem follows for $l
\geq 1.$ For the case $l<1$ we can consider
$w(x,t)=\underline{u}(x,t) - \overline{u}(x,t)$ and prove the
theorem in a similar way using the positiveness of a subsolution
or a supersolution.
\end{proof}

\begin{corollary}\label{Cor1}
Problem (\ref{v:u})--(\ref{v:n}) has a maximal solution in $Q_T$
for small values of  $\,T.$
\end{corollary}
\begin{proof}
In the previous section we prove the existence of a local solution
$u_m(x,t)=\lim_{\varepsilon \to 0} u_\varepsilon(x,t).$ Let
$v(x,t)$ be any other solution of (\ref{v:u})--(\ref{v:n}) in
$Q_T.$ By Theorem~\ref{Th3} we have $u_\varepsilon(x,t) \geq
v(x,t).$ Taking $\varepsilon \to 0$, we conclude $u_m(x,t) \geq
v(x,t).$ Obviously, $u_m(x,t)$ is a maximal solution of
(\ref{v:u})--(\ref{v:n}) in $Q_T.$
\end{proof}
Next lemma shows the positiveness of all nontrivial solutions for
$t>0$ if $p \geq 1.$
\begin{lemma}\label{Lem1}
Let $u_0$ is a nontrivial function in $\Omega,$  $p \geq 1$ or
$c(x,t) \equiv 0.$ Suppose $u$ is a solution of
(\ref{v:u})--(\ref{v:n}) in $Q_T.$ Then $u>0$ in ${Q}_T \cup S_T.$
\end{lemma}
\begin{proof}
As $c(x,t)$ and $u(x,t)$ are continuous in $\overline{Q}_T$
functions then we have
\begin{equation}\label{E:1.10}
\max \{c(x,t), u(x,t)\}  \leq M, \,\, (x,t) \in \overline{Q}_T
\end{equation}
with some positive constant $M.$ Now we put $v= u\exp (\lambda t)$
where $\lambda \geq M^p.$ It is easy to verify that $v_t - \Delta
v \geq 0.$ Since $v (x,0) = u_0(x)\not\equiv0$ in $\Omega$ and $v
(x,t) \geq 0$ in $Q_T,$ by the strong maximum principle $v(x,t)>0$
in $Q_T.$ Let $v(x_0,t_0)=0$ in some point $(x_0,t_0)\in S_T.$
Then according to Theorem~3.6 of~\cite{Hu} it yields $\partial
v(x_0,t_0)/\partial\nu < 0,$ which contradicts~(\ref{v:g}).
\end{proof}

As a simple consequence of Theorem~\ref{Th3} and Lemma~\ref{Lem1},
we get the following uniqueness result for problem
(\ref{v:u})--(\ref{v:n}).
\begin{theorem}\label{Th4}\ Let problem (\ref{v:u})--(\ref{v:n}) has
a solution in $Q_T$ with nonnegative initial data for $l \geq 1$
and with positive initial data under the conditions $l < 1,$ $p
\geq 1$ or a positive in ${Q}_T\cup \Gamma_T$ solution if $\max
(p,l) < 1.$ Then a solution of (\ref{v:u})--(\ref{v:n}) is unique
in $Q_T.$
\end{theorem}

Now we shall prove the nonuniqueness of a solution of our problem
with trivial initial datum for $l<p.$ We note that problem
(\ref{v:u})--(\ref{v:n}) with trivial initial datum has trivial
solution.

\begin{theorem}\label{Th5}\ Let $l < min \{1, p\}$ and $u_0(x) \equiv
0.$  Suppose that
\begin{equation}\label{E:1.101}
k(x,y_0,t_0)>0 \,\,\, \textrm{for any} \,\,\, x \in
\partial\Omega \,\,\, \textrm{and some} \,\,\, y_0 \in
\partial \Omega \,\,\, \textrm{and} \,\,\, t_0 \in [0,T).
\end{equation}
Then maximal solution $u_m (x,t)$ of problem
(\ref{v:u})--(\ref{v:n}) is nontrivial function in $Q_T.$
\end{theorem}
\begin{proof} As we showed in Theorem~\ref{Th2} and Corollary~\ref{Cor1} a
maximal solution $u_m(x,t) = \lim_{\varepsilon \to 0}
u_\varepsilon (x,t),$ where $u_\varepsilon (x,t)$ is some positive
in $\overline{Q}_T$ supersolution of (\ref{v:u})--(\ref{v:n}). To
prove theorem we construct a nontrivial nonnegative subsolution
$\underline{u} (x,t)$ of (\ref{v:u})--(\ref{v:n}) with trivial
initial datum. By Theorem~\ref{Th3} then we have $u_\varepsilon
(x,t)\geq \underline{u} (x,t)$ and therefore maximal solution
$u_m(x,t)$ is nontrivial function.

To construct a subsolution we use the change of variables in a
neighborhood of $\partial \Omega$ as in \cite{CPE}. Let $\overline
x$ be a point in $\partial \Omega.$ We denote by $\widehat{n}
(\overline x)$ the inner unit normal to $\partial \Omega$ at the
point $\overline x.$ Since $\partial \Omega$ is smooth it is well
known that there exists $\delta >0$ such that the mapping $\psi
:\partial \Omega \times [0,\delta] \to \mathbb{R}^n$ given by
$\psi (\overline x,s)=\overline x +s\widehat{n} (\overline x)$
defines new coordinates ($\overline x,s)$ in a neighborhood of
$\partial \Omega$ in $\overline\Omega.$

A straightforward computation shows that, in these coordinates,
$\Delta$ applied to a function $g(\overline x,s)=g(s),$ which is
independent of the variable $\overline x,$ evaluated at a point
$(\overline x,s)$ is given by
\begin{equation}\label{E:1.102}
\Delta g(\overline x,s) = \frac{\partial^2g}{\partial s^2}
(\overline x,s) - \sum_{j=1}^{n-1} \frac{H_j (\overline x)}{1-s
H_j (\overline x)}\frac{\partial g}{\partial s} (\overline x,s),
\end{equation}
where $H_j (\overline x)$ for $j=1,...,n-1,$ denotes the principal
curvatures of $\partial\Omega$ at $\overline x.$

Under the assumptions of the theorem there exists $\overline{t}>0$
such that $k(x,y,t)>0$ for $t_0 \leq t \leq t_0 + \overline{t},$
$x \in
\partial\Omega$ and $y \in V(y_0),$ where $V(y_0)$ is some
neighborhood of $y_0$ in $\overline{\Omega}.$

 Let $1/(1-l) < \alpha \leq 1/(1-p)$ for $p<1$ and $\alpha > 1/(1-l)$
for $p\geq1,$ $2< \beta < 2/(1-p)$ for $p < 1$ and $\beta > 2$ for
$p\geq1$ and assume that $A>0,$ $0<\xi_0 \leq 1$ and $0<T_0
<\min(T-t_0,\overline{t}, \delta^2).$ For points in $\partial
\Omega \times [0, \delta]\times (t_0,t_0+T_0]$ of coordinates
$(\overline x,s,t)$ define
\begin{equation}\label{E:1.11}
\underline{u}(\overline x,s,t)= A (t-t_0)^\alpha
\left(\xi_0-\frac{s}{\sqrt{t-t_0}}\right)^\beta_+
\end{equation}
and extend $\underline{u}$ as zero to the whole of
$\overline{Q_{\tau}}$ with $\tau = t_0+T_0$ . Using
(\ref{E:1.102}), we get that
\begin{eqnarray*}
&&\hspace{-0.6cm}\underline{u}_t (\overline x,s,t) - \Delta
\underline{u}(\overline x,s,t) + c(x,t)\underline{u}^p(\overline
x,s,t)=\alpha A (t-t_0)^{\alpha-1}\left(\xi_0-\frac{s}{\sqrt{t-t_0}}\right)^\beta_+\\
&+&\frac{\beta}{2} A
s(t-t_0)^{\alpha-3/2}\left(\xi_0-\frac{s}{\sqrt{t-t_0}}\right)^{\beta
- 1}_+ -\beta (\beta -1) A (t-t_0)^{\alpha-1}
\left(\xi_0-\frac{s}{\sqrt{t-t_0}}\right)^{\beta -2}_+\\
&-& \beta A
(t-t_0)^{\alpha-1/2}\left(\xi_0-\frac{s}{\sqrt{t-t_0}}\right)^{\beta
- 1}_+\sum_{j=1}^{n-1} \frac{H_j(\overline x)}{1-s H_j (\overline
x)} \\
&+& A^p c(x,t)(t-t_0)^{\alpha p}
\left(\xi_0-\frac{s}{\sqrt{t-t_0}}\right)^{\beta p}_+ \leq 0
    \end{eqnarray*}
for $(\overline x,s,t) \in \partial \Omega \times (0,
\delta]\times (t_0,\tau)$ and small values of $\xi_0.$

It is obvious,
    \begin{equation*}
\frac{\partial\underline u}{\partial\nu}(\overline
x,0,t)=-\frac{\partial\underline u}{\partial s}(\overline
x,0,t)=\beta A (t-t_0)^{\alpha-\frac{1}{2}}\xi_0^{\beta - 1}, \,\,
\overline x \in \partial \Omega, \, t_0 < t < \tau.
    \end{equation*}
To prove that $\underline{u}$ is the subsolution of
(\ref{v:u})--(\ref{v:n}) in $Q_\tau$ it is enough to check the
validity of the following inequality
\begin{equation}\label{E:1.12}
\beta A (t-t_0)^{\alpha-\frac{1}{2}}\xi_0^{\beta - 1}  \leq A^l
(t-t_0)^{\alpha l} \int_{\partial\Omega\times[0,\delta]}
k(x,(\overline{y},s),t) |J(\overline y,s)|
\left(\xi_0-\frac{s}{\sqrt{t-t_0}}\right)^{\beta l}_+ \,
d\overline y \, ds
\end{equation}
for $x \in \partial\Omega$ and $t_0 < t< \tau.$ Here $J(\overline
y,s)$ is Jacobian of the change of variables. Estimating the
integral $I$ in the right-hand side of (\ref{E:1.12})
$$
I = (t-t_0)^\frac{1}{2} \int_{\partial\Omega} \, d\overline y \,
\int_0^{\xi_0} k(x,(\overline{y},z\sqrt{t-t_0}),t) |J(\overline
y,z\sqrt{t-t_0})| \left( \xi_0-z \right)^{\beta l}_+ dz \geq C
(t-t_0)^\frac{1}{2},
$$
where positive constant $C$ does not depend on $t,$ we obtain that
(\ref{E:1.12}) is true if we take $T_0$ sufficiently small.
\end{proof}

\begin{remark}\label{Rem1}
Let $c(x,t_0) \leq c_0$ and $k(x,y,t_0) \geq k_0$ for $x \in
\Omega, \,$ $y \in \partial\Omega, \,$ some $t_0 \in [0, T)$ and
positive constants $c_0$ and $k_0.$ Then nonuniqueness of trivial
solution for our problem holds for $\, l=p < 1$ and large values
of $k_0/c_0.$ To prove this we can take in (\ref{E:1.11}) $\,
\alpha = 1/(1-l),\,$ $\beta = 2/(1-l)$ and $A$ in a suitable way.
\end{remark}

To prove the uniqueness of trivial solution of
(\ref{v:u})--(\ref{v:n}) with $u_0(x) \equiv 0$ we need the
following comparison principle.
\begin{lemma}\label{Lem2} Let $\underline{u}$ and $\overline{u}$
be a subsolution and a strict supersolution of
(\ref{v:u})--(\ref{v:n}) in ${Q}_T,$ respectively. Then
$\overline{u} \geq \underline{u}$ in $\overline{Q}_T.$
\end{lemma}
\begin{proof}
Let $v=\overline{u} + \varepsilon$ for $\varepsilon >0.$ Note that
for any $T_0 < T$ under suitable choice of $\varepsilon$ we have
$$
\frac{\partial v(x,t)}{\partial \nu} \geq \int_{\Omega} k(x,y,t)
v^l(y,t) dy \,\,\, \textrm{for} \,\,\, x \in
\partial \Omega, \,\,\, 0 \leq t \leq T_0.
$$
Now we can apply Theorem~\ref{Th3} to get the inequality $v \geq
\underline{u}$ in $\overline{Q}_{T_0}.$ Passing here to the limit
as $\varepsilon \to 0$ we prove the lemma. \end{proof}

\begin{theorem}\label{Th6}\ Let $p < l < 1$ and $u_0(x) \equiv
0.$  Suppose that  $c(x,t)\geq c_1>0$ for $(x,t) \in \overline
Q_T.$  Then the solution $u \equiv 0$ of problem
(\ref{v:u})--(\ref{v:n}) is unique in $Q_T.$
\end{theorem}
\begin{proof} Since there is the comparison principle for a solution
and a strict supersolution of (\ref{v:u})--(\ref{v:n}), it is
sufficiently to construct arbitrarily small strict supersolutions
which have a positive values on $\partial\Omega$ (see \cite{CER}
for another problem).

We shall use the change of variables in a neighborhood of
$\partial\Omega$ which we introduced in Theorem~\ref{Th5}. Let
$(1-l)/2 < \gamma < (1-p)/2, \,$
$0<\varepsilon<\delta^{1/\gamma},$ $A>0$ and $0<\xi_0 \leq 1.$
 For points in $\partial \Omega \times
[0, \delta]\times [0,T]$ of coordinates $(\overline x,s,t)$ define
\begin{equation}\label{E:1.12'}
\overline{u}_\varepsilon (\overline x,s,t)= \varepsilon A
\left(\xi_0- \varepsilon^{-\gamma} s\right)^{1/\gamma}_+
\end{equation}
and extend $\underline{u}_\varepsilon$ as zero to the whole of
$\overline{Q_T}.$ Using (\ref{E:1.102}), we get that
$$
\overline{u}_{\varepsilon t} (\overline x,s,t) - \Delta
\overline{u}_\varepsilon (\overline x,s,t) + c(x,t)
\overline{u}_\varepsilon^p (\overline x,s,t) \geq 0
$$
in $Q_T$ for small values of $\xi_0.$

To show that $\overline{u}_\varepsilon$ is a strict supersolution
we need to prove the following inequality
\begin{equation}\label{E:1.13}
\frac{A}{\gamma}\varepsilon^{1-\gamma} \xi_0^{(1-\gamma)/\gamma}
> (\varepsilon A)^l \int_{\partial\Omega\times[0,\delta]}
k(x,(\overline{y},s),t) |J(\overline y,s)|
\left(\xi_0-\varepsilon^{-\gamma} s \right)^{l/\gamma}_+ \,
d\overline y \, ds
\end{equation}
for $x \in \partial\Omega$ and $0 \leq t < T.$ Let $ k(x,y,t) \leq
k_1$ in $Q_T.$ Then estimating the right-hand side of
(\ref{E:1.13}) $J,$ we get
$$
J \leq k_1 C A^l \varepsilon^{l+\gamma} \xi_0^{(l+\gamma)/\gamma},
$$
where positive constant $C$ depend only on $p,$ $l$ and
$\partial\Omega.$ Hence (\ref{E:1.13}) holds if we take $\xi_0$
small enough.

We are in a position to complete the proof. Suppose for a
contradiction that there exists a solution of
(\ref{v:u})--(\ref{v:n}) with trivial initial function which is
not identically zero in ${Q}_T.$ Then by Lemma~\ref{Lem2} it
follows that
$$
u (x,t) \leq u_\varepsilon (x,t) \,\,\, \textrm{in} \,\,\, {Q}_T
$$
for all $0<\varepsilon<\delta^{1/\gamma}.$ This is a contradiction
since $u_\varepsilon \to 0$ uniformly on $\overline{Q}_T$ as
$\varepsilon \to 0.$ \end{proof}

\begin{remark}\label{Rem3}
Let the assumptions of Theorem~\ref{Th6} fulfill but only $l = p.$
Then the conclusion of Theorem~\ref{Th6} holds for large values of
$\, c_1/k_1,$ where $k_1$ was defined in the proof. To prove this
we can take in (\ref{E:1.12'}) $\, \gamma = (1-l)/2$ and $A$ in a
suitable way.
\end{remark}

\begin{remark}\label{Rem4}
Note that under the assumptions of Theorem~\ref{Th6} or
Remark~\ref{Rem3} there exists a class of nontrivial initial data
such that $u(x,t)\equiv 0$ for large values of $t.$ To prove this
we can modify function $\overline{u}_\varepsilon$ from
(\ref{E:1.12'}) in the following way
$$
\overline{u}_\varepsilon (\overline x,s,t)= \varepsilon
\left(\xi_0-\varepsilon^{-\gamma} s - \mu t \right)^{1/\gamma}_+,
\; \mu > 0.
$$
\end{remark}

\end{document}